\documentclass[12pt]{amsart}
\usepackage{amssymb,amsfonts,amsmath,amsopn,amstext,amscd,latexsym}
\usepackage[all,cmtip]{xy}
\usepackage{hyperref}
\theoremstyle{plain}

\theoremstyle{remark}

\DeclareMathOperator{\Hom}{Hom}

\DeclareMathOperator{\Gal}{Gal}
\DeclareMathOperator{\Aut}{Aut}

\DeclareMathOperator{\ad}{ad}

\DeclareMathOperator{\Ext}{Ext}

\DeclareMathOperator{\End}{End}
\DeclareMathOperator{\tr}{tr}

\DeclareMathOperator{\Lie}{Lie}

\DeclareMathOperator{\Fr}{Fr}

\DeclareMathOperator{\GL}{GL} 
\DeclareMathOperator{\SL}{SL}

\newcommand{\cD}{{\mathcal D}}

\newcommand{\cL}{{\mathcal L}}

\newcommand{\cT}{{\mathcal T}}

\newcommand{\frg}{{\mathfrak g}}

\newcommand{\bbF}{{\mathbb F}}
\newcommand{\bbG}{{\mathbb G}}

\newcommand{\bbQ}{{\mathbb Q}}

\newcommand{\bbZ}{{\mathbb Z}}

\begin{document}

\author[R. Guralnick]{Robert Guralnick}
\address{Department of Mathematics, University of Southern California, Los Angeles, CA 90089-2532, USA}
\email{guralnic@usc.edu}
\author[F. Herzig]{Florian Herzig}
\address{Institute for Advanced Study, Einstein Drive, Princeton, NJ 08540, USA}
\email{herzig@math.ias.edu}
\author[R. Taylor]{Richard Taylor}
\address{Department of Mathematics, Harvard University, 1 Oxford Street, Cambridge, MA 02138, USA}
\email{rtaylor@math.harvard.edu}
\author[J. Thorne]{Jack Thorne}
\address{Department of Mathematics, Harvard University, 1 Oxford Street, Cambridge, MA 02138, USA}
\email{thorne@math.harvard.edu}
\title{Appendix: Adequate subgroups}

\thanks{The first author was partially supported by NSF grants DMS-0653873 and DMS-1001962. The second author was partially
  supported by NSF grant DMS-0902044 and agreement DMS-0635607.  The third author was partially supported by NSF grant
  DMS-0600716 and by the IAS Oswald Veblen and Simonyi Funds. The second and third authors are grateful to the IAS for its
  hospitality during some of the work on this appendix.}

\maketitle
\thispagestyle{myheadings}
\setcounter{page}{60}
\theoremstyle{plain} 
\newtheorem{lm}{Lemma}
\newtheorem{prop}[lm]{Proposition}
\newtheorem{thm}[lm]{Theorem}

\renewcommand{\labelenumi}{(\roman{enumi})}

\newcommand{\fl}{\bbF_l}
\newcommand{\flb}{\overline{\bbF}_l}
\newcommand{\wt}{\widetilde}
\renewcommand{\o}{\overline}
\newcommand{\ang}[1]{\langle #1 \rangle}
\newcommand{\op}{^\mathrm{op}}
\renewcommand{\ss}{^\mathrm{ss}}
\newcommand{\mar}[1]{\marginpar{\tiny #1}}
\renewcommand{\(}{\textup{(}}
\renewcommand{\)}{\textup{)}}
\newcommand{\CC}{\textup{(C)}}

Let $l$ be a prime, and let $\Gamma$ be a finite subgroup of $\GL_n(\flb) = \GL(V)$. 
With these assumptions we say that \emph{Condition~\CC\ holds} if for every irreducible $\Gamma$-submodule $W \subset \ad^0 V$ there
exists an element $g \in \Gamma$ with an eigenvalue $\alpha$ such that $\tr e_{g, \alpha} W \neq 0$.
Here, $e_{g,\alpha}$ denotes the projection to the generalised $\alpha$-eigenspace of~$g$.
This condition arises in the definition of adequacy in section~2.

Let $\Gamma\ss$ denote the subset of $\Gamma$ consisting of the elements that are semisimple (i.e.\ of order prime to~$l$).

\begin{lm} Suppose that $\Gamma$ acts irreducibly on $V$. The following are equivalent.
  \begin{enumerate}
  \item Condition \CC.
  \item For every irreducible submodule $W \subset \ad^0 V$ there exists $g \in \Gamma\ss$ and $\alpha \in \flb$
    such that $\tr e_{g, \alpha} W \neq 0$.
  \item The set $\Gamma\ss$ spans $\ad V$ as an $\flb$-vector space.
  \end{enumerate}
\end{lm}

\begin{proof}
  Note that for any $g \in \Gamma$, $\Gamma$ contains both its semisimple and unipotent parts $g_s$ and $g_u$, respectively.
  (They are powers of $g$, as we work over $\flb$.)
  Since $e_{g,\alpha} = e_{g_s,\alpha}$ for all $g \in \Gamma$, the first two conditions are equivalent.

  To show that the last two conditions are equivalent, let $Z \subset \ad V$ be the span of the semisimple elements
  in~$\Gamma$.  Let $U$ denote the annihilator of $Z$ under the (non-degenerate, $\Gamma$-invariant) trace pairing:
  \begin{align}
    U &= \{ w \in \ad V : \tr (gw) = 0 \quad \forall g \in \Gamma\ss \} \label{eq:1} \\
    &= \{ w \in \ad V : \tr (e_{g,\alpha}w) = 0 \quad \forall g \in \Gamma\ss,\ \alpha \in \flb \}, \label{eq:2}
  \end{align}
  where we used that $e_{g,\alpha}$ is a polynomial in~$g$ and that $g = \sum \alpha e_{g,\alpha}$ for $g$ semisimple.
  
  Note that $U \subset \ad^0 V$ by taking $g = 1$ in~\eqref{eq:1}. From~\eqref{eq:2} it thus follows that the second
  condition is equivalent to $U = 0$.  Equivalently, $Z = \ad V$, which is the third condition.
\end{proof}

\begin{lm}\ 
  \begin{enumerate}
  \item
    Suppose that $\Gamma$ acts irreducibly on $V$. Condition \CC\ holds whenever $\Gamma$ has order prime to~$l$.
  \item
    Suppose that $V$, $V'$ are finite-dimensional vector spaces over $\flb$ and that
    $\Gamma \subset \GL(V)$, $\Gamma' \subset \GL(V')$ are finite subgroups that act irreducibly.
    If they both satisfy \CC, then the image of $\Gamma \times \Gamma'$ in $\GL(V \otimes V')$ also satisfies \CC.
\end{enumerate}
\end{lm}

\begin{proof}
  By Burnside's theorem, $\Gamma$ spans $\ad V$. If $\Gamma$ has order prime to~$l$, then every element is
  semisimple, so the lemma above applies.

  The second part of the proposition follows on noting that if $g$, $h$ are semisimple elements then $g \otimes h$ is
  semisimple, and appealing to the third characterization of condition \CC\ in the lemma above.
\end{proof}

Next we establish some preliminary results to prepare for our main theorem.

\begin{lm}\label{lm:torus}
  Suppose that $T$ is a torus over $\fl$. Let $X^* = X^*(T_{/\flb})$ and $X_* = X_*(T_{/\flb})$.
  There is a natural action of Frobenius $\Fr$ as an automorphism of $X^*$ and $X_*$.
  Suppose that $\Delta_* \subset X_*$ is a finite subset that is stable under the action of $\Fr$ and spans $X_* \otimes \bbQ$.
  \begin{enumerate}
  \item If $\mu \in X^*$ with $|\ang{\mu,\delta}|<l-1$ for all $\delta \in \Delta_*$ then $\mu(T(\fl))$ is trivial iff $\mu = 0$.
  \item If $V$ is a $T_{/\flb}$-module and all the weights $\mu$ of~$T_{/\flb}$ on~$V$ satisfy $|\ang{\mu,\delta}|<(l-1)/2$ for
    all $\delta \in \Delta_*$ then the $\flb$-span of $T(\fl)$ in $\ad V$ equals the $\flb$-span of $T(\flb)$.
\end{enumerate}
\end{lm}

\begin{proof}
  We can identify $\Hom(T(\fl),\overline{\bbF}_l^\times)$ with $X^*/(l-\Fr)X^*$.
  To prove the first part, suppose that $|\ang{\mu,\delta}|<l-1$ for $\delta \in
  \Delta_*$ and that $\mu(T(\fl))$ is trivial, so $\mu=(l-\Fr)\lambda$. Choose $\delta_1$ in $\Delta_*$ with
  $|\ang{\lambda,\delta_1}|$ maximal. If $\ang{\lambda, \delta_1} \ne 0$ then 
  \begin{equation*}
    l-1>|\ang{\mu,\delta_1}| \geq l|\ang{\lambda,\delta_1}|-|\ang{\lambda,\Fr^{-1}\delta_1}|
    \geq (l-1)|\ang{\lambda,\delta_1}| \geq l-1,
  \end{equation*}
  a contradiction. Therefore $\ang{\lambda, \delta_1} = 0$, so $\lambda=0$ and $\mu=0$.  In particular we see that if $\mu_1$
  and $\mu_2$ are two elements of $X^*$ with $|\ang{\mu_i,\delta}|<(l-1)/2$ for $\delta \in \Delta_*$ and $i = 1$, $2$ then
  $\mu_1|_{T(\fl)}=\mu_2|_{T(\fl)}$ iff $\mu_1 = \mu_2$. The second part now follows since both subspaces of $\ad V$ equal the
  $\flb$-linear span of the $T_{/\flb}$-equivariant projectors onto the weight spaces of~$T_{/\flb}$ in~$V$.
\end{proof}

\begin{lm}\label{lm:central-isogeny}
  Suppose that $G$ is a connected simply connected semisimple algebraic group over $\flb$ and $\phi : G \to \GL(V)$ a
  finite-dimensional representation. Let $G \supset B \supset T$ denote a Borel and maximal torus, and suppose that
  $|\ang{\mu_1-\mu_2,\alpha^\vee}|<l$ for all weights $\mu_1$, $\mu_2$ of $T$ on $V$ and all simple roots $\alpha$.
  Then there exist connected simply connected semisimple algebraic subgroups $I$ and $J$ of~$G$ such that $G = I \times J$,
  $\phi(J) = 1$, and $\phi$ induces a central isogeny of~$I$ onto its image $\o I$, which is a semisimple algebraic group.
\end{lm}

\begin{proof}
  Let $J$ denote the connected component of the kernel of $\phi$ with its reduced scheme structure. Then $J$ is smooth
  (\cite{Mil10}, Proposition I.5.18). By Theorem 8.1.5 of \cite{Spr09} and its proof, $J$ is semisimple and there is a second
  semisimple algebraic group $I \subset G$ which commutes with $J$ and such that $I \times J \rightarrow G$ is a central
  isogeny.  It follows from the simply-connectedness of~$G$ that it is an isomorphism of $I \times J$ onto $G$.  In
  particular, $I$ and $J$ are simply connected. Note that $T = T_I \times T_J$ and that $B = B_I \times B_J$ where
  $(B_I,T_I)$ (resp.\ $(B_J,T_J)$) is a Borel and maximal torus in $I$ (resp.\ $J$). (This follows from the fact that any
  smooth connected soluble subgroup of (resp.\ torus in) $G$ is conjugate to a subgroup of $B$ (resp.\ $T$).) Moreover $U =
  U_I \times U_J$, where $U$ denotes the unipotent radical of $B$. Let $\o{I}$ denote the image of $I$ under $\phi$. Then
  $\o{I}$ is again reduced and connected and hence also smooth.  In fact it is semisimple.  (See Proposition 14.10(1)(c) of
  \cite{Bor91}.) The map $\phi$ factors through an isogeny $I \rightarrow \o{I} \subset \GL(V)$.
  Let $\o B$, $\o T$, $\o U$ denote the images of $B_I$, $T_I$, $U_I$ in $\overline{I}$. Then these are all reduced and hence
  smooth. Moreover $\o T$ is a torus, $\o B$ is connected and soluble, $\o U$ is connected unipotent and $\o B=\o T \o U$. As
  $\dim \o I = \dim I = \dim T_I + 2 \dim U_I = \dim \o T + 2 \dim \o U$ we see that $\o B$ must be a Borel subgroup of~$\o
  I$ with unipotent radical $\o U$ and that $\o T$ is a maximal torus in $\o I$. The isogeny $I \rightarrow \o I$ induces an
  $l$-morphism from the root datum of~$\o I$ to the root datum of~$I$. (See section 9.6.3 of~\cite{Spr09}.) Then $I
  \rightarrow \o I$ is a central isogeny, as otherwise $T$ would have a weight occurring in $\Lie \o I \subset \ad V$ of the form $l \mu$ with
  $\mu$ non-zero and this would contradict our assumption on the weights of~$T$ on~$V$.
\end{proof}

Suppose that we are given $\flb$-vector spaces $W_i$ with $\dim W_i \le l$ for $i = 1$, \dots, $r$.
Then the maps
\begin{align*}
  \exp : X &\mapsto 1+X+\frac{X^2}{2!}+ \cdots + \frac{X^{l-1}}{(l-1)!} \\
  \log : 1+u &\mapsto u - \frac{u^2}{2} + \frac{u^3}{3} \pm \cdots - \frac{u^{l-1}}{l-1}
\end{align*}
define inverse bijections between the set of nilpotent elements in $\prod \End(W_i)$ and the set of unipotent elements in $\prod \GL(W_i)$.

\begin{lm}\label{lm:exp-and-log}
  Suppose that $G \subset \prod \GL(W_i)$ is a connected reductive group over~$\flb$ with $\dim W_i \le l$ for all $i$.
  Let $T$ be a maximal torus and $U$ be the unipotent radical of a Borel subgroup of $G$ that contains $T$. Suppose
  that $|\ang{\mu_1-\mu_2,\alpha^\vee}|<l$ for all weights $\mu_1$, $\mu_2$ of $T$ on $V = \bigoplus W_i$ and
  all simple roots $\alpha$.
  \begin{enumerate}
  \item The maps $\exp$ and $\log$ induce inverse isomorphisms of varieties between $\Lie U \subset \End(V)$ and $U \subset \GL(V)$.
  \item For any positive root $\alpha$ we have $\exp(\Lie U_\alpha) = U_\alpha$.
  \item The map $\exp : \Lie U \to U$ depends only on $G$ and $U$, but not on $V$, $W_i$, or the representation $G
    \hookrightarrow \GL(V)$.
  \item If $\theta$ is an automorphism of $G$ that preserves $T$ and $U$, then we have a commutative diagram:
    \begin{equation*}
      \xymatrix{\Lie U \ar[r]^{d\theta} \ar[d]_{\exp} & \Lie U \ar[d]^{\exp} \\ U \ar[r]^\theta & U }
    \end{equation*}
  \end{enumerate}
\end{lm}

\begin{proof}
  By the Lie--Kolchin theorem we may suppose $U$ is contained in the group $U'=\prod U_i'$, where $U'_i$ denotes the
  unipotent radical of a Borel subgroup of $\GL(W_i)$. The maps $\exp$ and $\log$ provide mutually inverse isomorphisms of
  varieties between $U'$ and $\Lie U'$. It remains to show that $\exp \Lie U = U$.  Note that the product of any $l$ elements
  of $\Lie U'$ is zero. Thus the Zassenhaus formula (see~\cite{Magnus}, section IV) tells us that to check that $\exp \Lie U
  \subset U$ it suffices to check that for any root~$\alpha$ we have $\exp (\Lie U_\alpha) \subset U$. Let $x_\alpha : \bbG_a
  \to U_\alpha$ be the root homomorphism corresponding to~$\alpha$ and let $X_\alpha = dx_\alpha(1) \in \Lie U_\alpha$. 
  Then formula II.1.19(6) of \cite{Jan87} shows that for $a \in \flb$,
  \begin{equation}\label{eq:4}
    x_\alpha(a) = \sum_{n = 0}^{l-1} a^n \frac{X_\alpha^n}{n!} = \exp(aX_\alpha)
  \end{equation}
  in $\GL(V)$, on noting that for $n<l$ we have $X_{\alpha,n}=X_\alpha^n/n!$ while $X_{\alpha,n}$ acts trivially on $V$ for $n\geq l$.
  (This latter assertion follows from formula II.1.19(5) of \cite{Jan87} because $V_\lambda$ and $V_{\lambda+n \alpha}$
  cannot both be non-zero.) Now by the Baker--Campbell--Hausdorff formula (see section IV.8 in part I of \cite{Serre-LALG})
  and the fact that the product of any $l$ elements of $\Lie U'$ is zero we see that $\exp \Lie U$ is a subgroup of~$U$. As
  $U$ is connected and smooth and $\dim \Lie U \geq \dim U$ we deduce that $\exp \Lie U= U$. This proves the first two
  parts.

  The third part follows inductively from equation~\eqref{eq:4} and the Zassenhaus formula:
  fix a total order $<$ on the set of positive roots such that if $\alpha$, $\beta$, $\alpha+\beta$ are positive
  roots, then $\max(\alpha,\beta) < \alpha+\beta$. We induct on the positive root $\gamma$. Suppose that
  we know that $\exp$ depends only on $G$ and $U$ on the subspace $\bigoplus_{\alpha > \gamma} \Lie U_\alpha$.
  Then the same is true for $\exp(X+Y)$ for any $X \in \Lie U_\gamma$ and $Y \in \bigoplus_{\alpha > \gamma} \Lie U_\alpha$
  by the Zassenhaus formula. (Note that $[\Lie U_\alpha,\Lie U_\beta] \subset \Lie U_{\alpha+\beta}$
  whenever $\alpha$, $\beta$ are positive roots.) This completes the proof of the third part.
    
  The last part follows from the third part, by considering the representation $G \xrightarrow \theta G \hookrightarrow \GL(V)$.
\end{proof}

\begin{lm}\label{lm:autom-lie-alg}
  Suppose that $G$ is a connected simply connected semisimple algebraic group over $\flb$.  Suppose that $l > 3$ and that $G$
  has no simple factor isomorphic to $\SL_n$ with $l | n$. Let $\frg$ denote the Lie algebra of~$G$.  Then $\frg$ contains no
  non-trivial abelian ideal, and the natural map $\Aut(G) \to \Aut(\frg)$ is a bijection. Moreover, a connected normal
  subgroup of~$G$ is preserved by an automorphism $\theta \in \Aut(G)$ if and only if its Lie algebra is preserved by
  $d\theta \in \Aut(\frg)$.
\end{lm}

Here, $\Aut(G)$ (resp., $\Aut(\frg)$) denotes the abstract group of automorphisms of the algebraic group~$G$ (resp., its Lie
algebra $\frg$). In the proof we use Chevalley groups in the sense of Steinberg's Yale notes \cite{Steinberg-Yale}.

\begin{proof}
  The universal Chevalley group over $\flb$ constructed using the complex semisimple Lie algebra $\cL$ of the same root
  system as $G$ is an algebraic group isomorphic to $G$ (see \cite{Steinberg-Yale}, \S 5). (In the notation of
  \cite{Steinberg-Yale}, we can let $V$ be any representation whose weights span the weight lattice, so that $\cL_\bbZ
  \subset \cL$ is the $\bbZ$-lattice spanned by the fixed Chevalley basis $H_i$, $X_\alpha$; see Cor.~2 on p.~18 of
  \cite{Steinberg-Yale}.) In particular, $\frg \cong \cL_\bbZ \otimes \flb$ (by the remark on p.~64 of
  \cite{Steinberg-Yale}).  Write $G = \prod G_i$ as a product of almost simple simply connected algebraic groups and
  correspondingly $\frg = \bigoplus \frg_i$. Then $Z(\frg_i) = 0$ by our assumption on $l$ and $G$ (see Theorem~2.3 in
  \cite{Hurley}) and hence all $\frg_i$ are simple (\cite{Steinberg-auto}, 2.6(5)). Moreover $\frg_i \cong \frg_j$ implies
  $G_i \cong G_j$ (\cite{Steinberg-auto}, 8.1). The $G_i$ (resp., $\frg_i$) are uniquely characterised as the minimal
  non-trivial connected normal subgroups of $G$ (resp., minimal non-trivial ideals of $\frg$), so they are permuted by
  automorphisms. Therefore if $\Aut(G_i) \to \Aut(\frg_i)$ is a bijection for all $i$, then so is $\Aut(G) \to \Aut(\frg)$,
  and also the final claim of the proposition follows. (Note that any connected normal subgroup is a
  product of some of the $G_i$.) We can thus assume, without loss of generality, that $G$ is almost simple.

  Let $G^{\ad}$ denote the adjoint form of $G$. As $G$ is the universal cover of $G^{\ad}$ and as $G^{\ad} = G/Z(G)$, we have
  $\Aut(G) = \Aut(G^{\ad})$. As $Z(\frg) = 0$ we see that the natural map $\frg \to \Lie G^{\ad}$ is an isomorphism. Thus it
  suffices to show that $\Aut(G) = \Aut(\frg)$ whenever $G$ is simple of \emph{adjoint} type and $\frg = \Lie G$. Thus
  we write $G$ for $G^{\ad}$ from now on.
  
  As an algebraic group $G$ is isomorphic to the adjoint Chevalley group over $\flb$ (again by \cite{Steinberg-Yale}, \S
  5).  (In the notation of \cite{Steinberg-Yale}, we take $V$ to be the adjoint representation $\frg$.) Thus we can identify
  $G(\flb)$ with the subgroup of $\GL(\frg)$ generated by the elements $x_\alpha(t) := \exp(\ad(tX_\alpha))$, where $t \in
  \flb$ and $\alpha$ is any root.  As each $\ad(tX_\alpha)$ is a derivation of $\frg$, the group $G(\flb)$ is actually
  contained in $\Aut(\frg)$.  For any $\eta \in \Aut(\frg)$, we have $\eta \circ \ad X \circ \eta^{-1} = \ad(\eta X)$ in
  $\GL(\frg)$.  It follows that the natural action of $G(\flb) \subset \GL(\frg)$ on $\frg$ agrees with the adjoint action of
  $G(\flb)$ on $\frg \subset \End(\frg)$.

  The choice of Chevalley basis gives rise to a maximal torus $T$ and a Borel $B$ that contains it (\cite{Steinberg-Yale}, \S
  5).
    From Theorem 9.6.2 in~\cite{Spr09} we deduce the following, using that $G$ is adjoint.  For each symmetry $\pi$ of the
  Dynkin diagram $\cD$ there is a unique $\pi' \in \Aut(G)$ that preserves $(B,T)$ and that permutes the $x_{\alpha_i}(1) \in
  B$ according to $\pi$ (where $\alpha_i$ are the simple roots). Moreover, $\Aut(G)$ is the semidirect product of $G$ (acting
  by inner automorphisms) and $\Aut(\cD)$. Also, the elements of $\Aut(\cD)$ biject with the ``graph automorphisms'' of
  $\frg$ (\cite{Steinberg-auto}, \S 3).
 
   The result now follows from (\cite{Steinberg-auto}, 4.2 and 4.5), as the group $\mathfrak H$ in \cite{Steinberg-auto} is
  actually contained in $G(\flb)$ since $\flb$ is algebraically closed (see Lemma 19 on p.~27 of
  \cite{Steinberg-Yale}). (Note that the uniqueness statement in (\cite{Steinberg-auto}, 4.2) is incorrect and seems to be a
  typo.)
\end{proof}

The following proposition may be of independent interest. The proof uses the classification of finite simple groups.
Without it, the proof still goes through for $l$ sufficiently large (depending on~$d$ and ineffective) by appealing
to~\cite{LP} instead of~\cite{Gur99}.

\begin{prop}\label{prop:alg-rep}
  Suppose that $V$ is a finite-dimensional $\flb$-vector space and that $\Gamma \subset \GL(V)$ is a finite subgroup that
  acts semisimply on $V$. Let $\Gamma^0 \subset \Gamma$ be the subgroup generated by elements of $l$-power order. Then $V$
  is a semisimple $\Gamma^0$-module. Let $d \geq 1$ be the maximal dimension of an irreducible $\Gamma^0$-submodule of $V$.
  Suppose that $l \geq 2(d+1)$. Then there exists an algebraic group $G$ over $\fl$ and a semisimple representation
  $r : G_{/\flb} \to \GL(V)$ with the following properties:                                                                                             
  \begin{enumerate}
  \item The connected component $G^0$ is semisimple, simply connected.
  \item $G \cong G^0 \rtimes H$, where $H$ is a finite group of order prime to $l$.
  \item $r(G(\fl)) = \Gamma$.
  \end{enumerate}
  Moreover, if $T \subset G^0$ is a maximal torus and if $\mu$ is a weight of $T_{/\flb}$ on $V$ then
  $\sum |\ang{\mu,\alpha^\vee}| < 2d$, where $\alpha$ ranges over the roots of $G^0_{/\flb}$. 
  Also, $\Gamma$ does not have any composition factor of order~$l$.
\end{prop}

\begin{proof}
Write $V = \bigoplus_i {W_i}$ as a direct sum of irreducible $\Gamma^0$-modules. Since $\dim W_i \leq l$ for all $i$, we see
that every element of $l$-power order in the image of $\Gamma^0 \rightarrow \GL(W_i)$ actually has order dividing $l$. Since
$\Gamma^0 \hookrightarrow \prod \GL(W_i)$, we deduce that every element of $\Gamma^0$ of $l$-power order actually has order
dividing $l$. Note that $\Gamma/\Gamma^0$ has order prime to~$l$.

\newcounter{step}\setcounter{step}{0}

\addtocounter{step}{1} \emph{Step \arabic{step}. We show that there exists a connected simply connected semisimple algebraic
  group $G^0$ over $\fl$ and a finite central subgroup $Z_0 \subset G^0(\fl)$ with $G^0(\fl)/Z_0 \cong \Gamma^0$.} Let
$\Gamma_i$ denote the image of~$\Gamma^0$ in $\GL(W_i)$. Note that $\Gamma_i$ has no non-trivial normal subgroup of $l$-power
order (since $\Gamma_i$ acts faithfully on $W_i$, and an $l$-group acting on a non-zero $\flb$-vector space has non-zero fixed
points). So by Theorem B of \cite{Gur99}, $\Gamma_i$ is a central product of quasisimple Chevalley groups. (Note that if $l =
11$ then $\dim W_i <7$.) Now $\Gamma^0$ is a subgroup of $\prod \Gamma_i$ that surjects onto each factor, so $Z(\Gamma^0) =
\Gamma^0 \cap \prod Z(\Gamma_i)$.  Thus $\Gamma^0/Z(\Gamma^0)$ is a subgroup of $\prod \Gamma_i/Z(\Gamma_i)$, a product of
simple Chevalley groups, that surjects onto each factor.  By a theorem of Hall (Lemma 3.5 in \cite{Kup}), $\Gamma^0/Z(\Gamma^0)$ is itself
isomorphic to a direct product of simple Chevalley groups. It follows that $\Gamma^0 = [\Gamma^0,\Gamma^0] Z(\Gamma^0)$.
Since $\Gamma^0$ is generated by elements of order~$l$ and $Z(\Gamma^0)$ is of order prime to~$l$, it follows moreover that
$\Gamma^0$ is perfect. Therefore $\Gamma^0$ is a perfect central extension of a product $\prod H_j$ of simple Chevalley
groups $H_j$, so there exists a surjective homomorphism $\pi : \prod \wt H_j \to \Gamma^0$ with central kernel, where $\wt H_j$ is the
universal perfect central extension of~$H_j$.

As $l > 3$ (to rule out Suzuki and Ree groups) there exist connected simply connected algebraic groups ${G_j}$ over $\fl$
such that $H_j \cong G_j(\fl)/Z(G_j(\fl))$.  (Note that $G_j$ is the restriction of scalars of an absolutely almost simple
algebraic group over a finite extension of~$\fl$.)  Since $l > 3$ it is known that $\wt H_j \cong G_j(\fl)$ (see section 6.1
in \cite{GLS3}, particularly table 6.1.3). So we can take $G^0 = \prod G_j$ and $Z_0 = \ker \pi$.

Since $\Gamma^0/Z(\Gamma^0)$ is a product of nonabelian simple groups and since $Z(\Gamma^0)$ and $\Gamma/\Gamma^0$ are
of order prime to~$l$, it follows that $\Gamma$ does not have any composition factor of order~$l$.

Let $G^0 \supset B \supset T$ denote a Borel and maximal torus defined over~$\fl$.

\addtocounter{step}{1}\newcounter{action-of-alg-grp}\setcounter{action-of-alg-grp}{\value{step}}
\emph{Step \arabic{step}. We lift $V$ to a $G^0_{/\flb}$-module and compare the actions of~$T(\fl)$ and
  $T(\flb)$ on $V$.}  Let $U$ denote the unipotent radical of $B$ and set $N=N_{G^0}(T)$. Let $B\op$ denote the
opposite Borel subgroup to $B$ containing $T$ and let $U\op$ denote its unipotent radical. (See Theorem 14.1 of \cite{Bor91}.
By uniqueness we see it is defined over $\fl$.) Let $X=X^*(T_{/\flb})$ with its subset $\Phi$ of roots and $\Phi^+$
(resp.\ $\Delta$) the set of positive (resp.\ simple) roots corresponding to $B$.
Let $X^+ \subset X$ be the subset of dominant weights.
There is a semisimple algebraic action of $G^0_{/\flb}$ on $V$, say $\phi:G^0_{/\flb}\rightarrow \GL(V)$, such
that:
\begin{enumerate}
\item the highest weight $\lambda$ of a simple submodule is restricted (i.e.\ $0 \leq \ang{\lambda,\alpha^\vee}<l$ for all $\alpha \in \Delta$),
\item the action of $G^0(\fl)$ is the one induced by the map $G^0(\fl) \rightarrow \Gamma^0$,
\item the subspaces $W_i$ are $G^0_{/\flb}$-stable.
\end{enumerate}
(This follows from a result of Steinberg: see Theorem~2.11 in~\cite{Humphreys}. Note that \cite{Humphreys} works with an
algebraic group $\mathbf{G}$ that is simple, but the proof given does not depend on that assumption.)  By Proposition 3 of
\cite{Ser94} we see that if $\lambda$ in $X^+$ is a weight of~$T_{/\flb}$ on~$V$ then $\sum_{\alpha \in \Phi^+}
\ang{\lambda,\alpha^\vee} < d$; in particular, $\ang{\lambda,\alpha^\vee}<(l-1)/2$ for all $\alpha \in \Phi^+$. (Note that
$\dim W_i \leq (l-1)/2$ and that the proof of that proposition does not require that $G^0_{/\flb}$ be almost simple.)  If
$\mu$ is a weight of $T_{/\flb}$ on $V$ then we see that there is $w$ in the Weyl group with $w \mu \in X^+$ and $0\leq
\ang{w \mu, \alpha^\vee} <(l-1)/2 $ for all $\alpha \in \Phi^+$, and we deduce that $|\ang{\mu,\alpha^\vee}|<(l-1)/2$ for all
$\alpha \in \Phi$. We also deduce that if $\mu$ is a weight of~$T_{/\flb}$ on~$\ad V$ then $|\ang{\mu,\alpha^\vee}|<l-1 $ for
all $\alpha \in \Delta$.

\addtocounter{step}{1} \newcounter{the-group-I}\setcounter{the-group-I}{\value{step}} \emph{Step \arabic{step}. The
  semisimple group $\o I \subset \GL(V)$ and its simply connected cover $I \subset G^0_{/\flb}$.}  Since $|\ang{\mu,
  \alpha^\vee}|<l/2$ for all weights $\mu$ of $T_{/\flb}$ on $V$ and all $\alpha \in \Delta$ we may apply
Lemma~\ref{lm:central-isogeny} to $\phi : G^0_{/\flb} \rightarrow \GL(V)$. We obtain connected simply connected semisimple
algebraic subgroups $I$, $J$ of $G^0_{/\flb}$ such that $G^0_{/\flb} = I \times J$, $\phi(J) = 1$, and $\phi$ induces a
central isogeny of~$I$ onto its image $\o I$, which is a semisimple algebraic group.  Note that $T_{/\flb} = T_I \times T_J$
and that $B_{/\flb} = B_I \times B_J$ where $(B_I,T_I)$ (resp.\ $(B_J,T_J)$) is a Borel and maximal torus in $I$ (resp.\ 
$J$).  Moreover $U_{/\flb} = U_I \times U_J$.  Let $\o B$, $\o T$, $\o U$, $\o B\op$, $\o U\op$ denote the images of $B_I$,
$T_I$, $U_I$, $B_I\op$, $U_I\op$ in $\overline{I}$. Then $\o T$ is a maximal torus of $\o I$, and $\o B$, $\o B\op$ are
opposite Borel subgroups containing it. Also $\o U$, $\o U\op$ are the unipotent radicals of $\o B$, $\o B\op$. Since $I \to
\o I$ is a central isogeny, $U_I \rightarrow \o U$ and $U_I\op \rightarrow \o U\op$ are isomorphisms.

\addtocounter{step}{1} \newcounter{exp-and-log}\setcounter{exp-and-log}{\value{step}} \emph{Step \arabic{step}. The maps
  $\log$ and $\exp$ provide inverse isomorphisms of varieties between $\o U \subset \GL(V)$ and $\Lie \o U \subset \ad V$.}
This follows from Lemma~\ref{lm:exp-and-log} applied to $\o I \subset \GL(V)$ since $\dim W_i \le l$ for all $i$ and
$|\ang{\mu, \alpha^\vee}|<l/2$ for all weights $\mu$ of $T_{/\flb}$ on $V$ and all $\alpha \in \Delta$.
(Note that $T_I \to \o T$ induces a bijection on coroots since $I \to \o I$ is a central isogeny; thus
$T \to \o T$ induces a surjection on coroots.)

\addtocounter{step}{1} \newcounter{rational-log}\setcounter{rational-log}{\value{step}} \emph{Step \arabic{step}. The
  $\flb$-span of $\log U(\fl)$ is $\Lie \o U$.}  Since $d \phi: \Lie U \rightarrow \Lie \o U$ is surjective, it suffices
to show that there is an isomorphism $\log : U \to \Lie U$ defined over $\fl$ such that $d \phi \circ \log = \log\circ \phi$.
Pick an $\fl$-structure on $V$.  The map $G^0_{/\flb} \to \GL(V)$ can be defined over some~$\bbF_{l^s}$ and so taking
restrictions of scalars from $\bbF_{l^s}$ to $\fl$ we get an $\fl$-vector space $V'$ and a map $\psi : G^0 \to \GL(V')$. The
map $G^0_{/\flb} \to \GL(V)$ is obtained from $\psi$ by extending scalars to $\flb$ and projecting to a direct summand $V$ of
$V' \otimes \flb$. The dimension of all irreducible factors of $V' \otimes \flb$ is at most $l$. Moreover for any weight
$\lambda$ of~$T_{/\flb}$ on~$V' \otimes \flb$ we have $|\ang{\lambda, \alpha^\vee}|<(l-1)/2$ for all $\alpha \in \Phi^+$.

By Lemma~\ref{lm:central-isogeny} we see that $\psi : G^0 \to \GL(V')$ is a central isogeny onto its image. (By construction we have
$(\ker \psi)(\fl) = Z_0$. Suppose that $\ker \psi$ is not finite. Then it has to contain one of the $\fl$-almost simple factors of
$G^0 = \prod G_j$. But $G_j(\fl)$ is nonabelian.)

In particular, $\psi$ induces an isomorphism $U \to \psi(U)$.
Then Lemma~\ref{lm:exp-and-log} (applied to the image of $\psi_{/\flb}$) gives the desired map $\log: U \to \Lie U \subset \ad V'$.

\addtocounter{step}{1}
\emph{Step \arabic{step}: Some properties of $G^0(\fl)$.} The pair $(B({\fl}),N({\fl}))$ is a split $BN$ pair in $G^0({\fl})$ (see
section 1.18 of \cite{Car93}). Also $U({\fl})$ is a Sylow $l$-subgroup of $G^0({\fl})$ and
$B({\fl})=N_{G^0({\fl})}(U({\fl}))=N_{G^0({\fl})}(B({\fl}))$ (see Proposition 2.5.1 of \cite{Car93}).

Moreover $T({\fl})$ is a Sylow $l$-complement in
$B({\fl})$. 
Note that $U\op({\fl})$ is $N({\fl})$-conjugate to $U({\fl})$. (The longest Weyl element $w_0$ is stable under Frobenius,
hence represented by an element $n_0 \in N(\fl)$. Then use that $U\op = n_0 U n_0^{-1}$.)  Moreover the second-last displayed
equation on page~74 (section 2.9) of \cite{Car93} shows that $U\op({\fl})$ is the unique $N({\fl})$-conjugate of $U({\fl})$
with trivial intersection with $U({\fl})$.

\addtocounter{step}{1}
\emph{Step \arabic{step}. We have $N({\fl})=N_{G^0({\fl})}(T({\fl}))$ so that
  $N_{G^0({\fl})}(T({\fl})) \cap N_{G^0({\fl})}(B({\fl})) = T({\fl})$ and $Z_0 \subset Z(G^0(\fl))
  \subset T(\fl)$.}

Suppose that $g$ is in $N_{G^0({\fl})}(T({\fl}))$. One can write
$g$ uniquely as $unu'$ where $u \in U({\fl}), n \in N({\fl})$ maps to $w_n$ in the Weyl group and $u' \in U_{w_n}$ in
the notation of Theorem 2.5.14 of \cite{Car93}. Then for any $h$ in $T({\fl})$ we can find $h'$ and $h''$ in $T({\fl})$
such that
\[hunu'=unu'h' \text{\quad and\quad } h''unu'=unu'h,\]
i.e.,
\[(huh^{-1})(hn)u'= u(nh')(h'^{-1}u'h')\]
and
\[(h''uh''^{-1})(h''n)u'=u(nh)(h^{-1}u'h).\]
As $T({\fl})$ normalizes $U({\fl})$ and $U_{w_n}$ and as
$w_{nh}=w_n=w_{hn}$ the uniqueness assertion of Theorem 2.5.14 of \cite{Car93} tells us that $huh^{-1}=u$ and $u'=h^{-1}u'h$.
Thus $u \in Z_{U({\fl})}(T({\fl}))$ and $u' \in Z_{U_{w_n}}(T({\fl})) \subset Z_{U({\fl})}(T({\fl}))$.
So it suffices to prove that $Z_{U(\flb)}(T({\fl}))={1}$. By Proposition 8.2.1 in \cite{Spr09}

it suffices to show that $Z_{U_\alpha(\flb)}(T({\fl}))={1}$ for all $\alpha \in \Phi^+$. By Proposition 8.1.1(i) in
\cite{Spr09}

it suffices that $\alpha$ is non-trivial on $T({\fl})$ for all $\alpha \in \Phi^+$.
As $l\geq 5$, this follows from Lemma~\ref{lm:torus}(i) (applied with $\Delta_*$ the set of simple coroots).

\addtocounter{step}{1} \emph{Step \arabic{step}. We find a subgroup $H$ of order prime to~$l$ such that $\Gamma = \Gamma^0 H$.}
Let $H$ denote the subgroup of $h \in \Gamma$ which normalize both the image of $B({\fl})$ and the image of $T({\fl})$
in $\Gamma^0$. Then by the previous paragraph we see that $H \cap \Gamma^0$ is $T({\fl})/Z_0$. Thus $H$ has order prime to~$l$.

Moreover if $\gamma \in \Gamma$ we see that $\gamma (B({\fl})/Z_0) \gamma^{-1}$ is the
normalizer of a Sylow $l$-subgroup of $G^0({\fl})/Z_0$ and hence $G^0({\fl})$-conjugate to $B({\fl})/Z_0$, say $\gamma
(B({\fl})/Z_0)\gamma^{-1}=k (B({\fl})/Z_0) k^{-1}$ with $k \in G^0({\fl})$. Then $k^{-1}\gamma
(T({\fl})/Z_0)\gamma^{-1}k$ is a Sylow $l$-complement in $B({\fl})/Z_0$ and hence (by Hall's theorem)
$B({\fl})/Z_0$-conjugate to $T({\fl})/Z_0$, say \[k^{-1}\gamma (T({\fl})/Z_0)\gamma^{-1}k = k'(T({\fl})/Z_0) k'^{-1}\] for
some $k' \in B({\fl})$. Then $(kk')^{-1}\gamma$ lies in $H$ and we deduce that $\Gamma$ is generated by $H$ and
$G^0({\fl})/Z_0 = \Gamma^0$.

\addtocounter{step}{1}\newcounter{lift-conj}\setcounter{lift-conj}{\value{step}} \emph{Step \arabic{step}. Lifting the
   conjugation action of~$H$ on~$\Gamma^0$ to $G^0$.}  We first show that $G^0_{/\flb}$ has no simple factor
$\SL_n$ with $l | n$ by showing that any such factor would act trivially on $V = \bigoplus W_i$, contradicting that
$G^0(\fl)/Z_0$ acts faithfully. So suppose that ${\SL_n}_{/\flb}$ has an irreducible module of dimension less than
$l-1$. Then by Proposition~3 in~\cite{Ser94} its highest weight $\lambda$ would satisfy $\sum \ang{\lambda,\alpha^\vee} <
l-1$, where $\alpha$ runs through the set of positive roots. A calculation shows that the left-hand side is at least $n-1$ if
$\lambda$ is non-zero. So if $n \ge l$, then $\lambda = 0$.

Next we claim that $d\phi : (\Lie G^0)(\flb) \to \ad V$ is injective on the subspace $(\Lie G^0)(\fl)$. Note first that it is
injective on $(\Lie U)(\fl)$ as $\phi$ is injective on $U(\fl)$. (Consider the isomorphism $\log : U(\fl) \to (\Lie U)(\fl)$
constructed in Step~\arabic{rational-log}.) Similarly $d\phi$ is injective on $(\Lie U\op)(\fl)$. Since $\phi$ maps $U$ to
$\o U$, $T$ to $\o T$, $U\op$ to $\o U\op$, and since $\Lie G^0 = \Lie U \oplus \Lie T \oplus \Lie U\op$, $\Lie \o I = \Lie
\o U \oplus \Lie \o T \oplus \Lie \o U\op$ it follows that the kernel of $d\phi$ on $(\Lie G^0)(\fl)$ is contained in $(\Lie
T)(\fl)$. But $(\Lie G^0)(\flb)$ contains no non-trivial abelian ideal by Lemma~\ref{lm:autom-lie-alg}.  This proves the
claim.

Note that $H$ acts by conjugation on $\GL(V)$ and $\ad V$, in particular it preserves the Lie algebra structure of $\ad V$.
By definition $H$ stabilises the image of $U(\fl)$ in $\GL(V)$ and hence by Step~\arabic{rational-log} it also stabilises
$\log U(\fl) = d\phi((\Lie U)(\fl))$. Because $U\op({\fl})$ is the unique $N_{G^0({\fl})}(T({\fl}))$-conjugate of $U({\fl})$
that has trivial intersection with $U({\fl})$, it is also stabilised by $H$. The previous argument
then shows that $H$ stabilises $d\phi((\Lie U\op)(\fl))$. Since $[\Lie U,\Lie U\op] = \Lie G^0$ (as we may check over
$\flb$), it follows that $H$ stabilises the image of $(\Lie G^0)(\fl)$ in $\ad V$. By extending scalars, we get a natural
action of $H$ on $(\Lie G^0)(\flb)$. This action lifts uniquely to an action on $G^0_{/\flb}$ by
Lemma~\ref{lm:autom-lie-alg}. 

We claim that with respect to the $H$-action on $G^0_{/\flb}$ just constructed, $\phi : G^0_{/\flb} \to \GL(V)$ is
$H$-equivariant.  We first show that the conjugation action of $H$ on $\GL(V)$ stabilises $\o I$.  If $h \in H$ then $h$
sends $U({\fl})$ to itself and hence $\log U({\fl})$ to itself and hence $\Lie \o U$ to itself and hence $\o U$ to itself.
Similarly $h$ stabilises $\o U\op$.  As the root subgroups generate $\o I$ (by Theorem~8.1.5 in \cite{Spr09}), we see that
$h$ indeed stabilises $\o I$. This action of~$H$ on~$\o I$ lifts uniquely to an action on the simply connected cover $I$
of~$\o I$. (For existence use Theorem 9.6.5 of \cite{Spr09} and the conjugation action of~$T_I$.  For uniqueness use the
semisimplicity of~$I$.) On the other hand, Lemma~\ref{lm:autom-lie-alg} shows that the $H$-action on $G^0_{/\flb}$ respects
the decomposition $G^0_{/\flb} = I \times J$. Since $J$ is killed by $\phi$ it suffices to show that the two $H$-actions on
$I$ (one coming from $\o I$ and one from $G^0_{/\flb}$) agree. By Lemma~\ref{lm:autom-lie-alg} we can check this on the Lie
algebra. The same lemma shows that $d\phi: \Lie I \to \Lie \o I$ is an isomorphism, since $\Lie I$ contains no non-trivial
abelian ideal. By construction both $H$-actions on $\Lie I$ are compatible with the $H$-action on $\Lie \o I$, so the two
$H$-actions on $I$ indeed agree. Therefore $\phi$ is $H$-equivariant. A fortiori, it extends to a homomorphism $G^0_{/\flb}
\rtimes H \to \GL(V)$.

Finally we show that the $H$-action on $G^0_{/\flb}$ descends to $G^0$. Suppose that $h \in H$ and $\sigma \in
\Gal(\flb/\fl)$. The automorphism $\sigma h\sigma^{-1} h^{-1}$ is trivial on $(\Lie G^0)(\fl)$, hence trivial
on $(\Lie G^0)(\flb)$, hence trivial on $G^0_{/\flb}$ by Lemma~\ref{lm:autom-lie-alg}. Therefore
the $H$-action indeed descends to $G^0$.

By construction, the image of $G^0(\fl) \rtimes H$ is $\Gamma$.
Let $G = G^0 \rtimes H$ and $r : G_{/\flb} \to \GL(V)$ the homomorphism we just obtained. 
It remains to show that $r$ is semisimple. But this follows from Lemma~5(b) in \cite{Ser94} since the restriction of $r$ to
$G^0_{/\flb}$ is semisimple and $(G:G^0)$ is prime to~$l$.
\end{proof}

We remark that for the purpose of proving Theorem~\ref{thm:main} we do not need an $H$-action on $G^0$, we only need an
$H$-action on $G^0_{/\flb}$ that is compatible with the $H$-action on $\GL(V)$. Since $G^0_{/\flb} = I \times J$, we can lift
the $H$-action on $\o I$ to $I$ as above and let $H$ act arbitrarily on $J$; for this it is not necessary to appeal to
Lemma~\ref{lm:autom-lie-alg}.

\begin{lm}\label{lm:ss-elts-span}
  Suppose that $G$ is a linear algebraic group over~$\flb$ such that the connected component $G^0$ is semi-simple and simply
  connected and such that $l$ does not divide $(G:G^0)$. Let $G^0\supset B\supset T$ denote a Borel subgroup and a maximal
  torus and let $\cT$ denote the normalizer of the pair $(B,T)$ in $G$. Then the $G^0(\flb)$-conjugates of $\cT(\flb)$ equal
  the semisimple elements of $G(\flb)$ and they are Zariski dense in $G$. In particular, if $V$ is an irreducible
  representation of $G$ then the $G^0(\flb)$-conjugates of $\cT(\flb)$ span $\ad V$ over~$\flb$.
\end{lm}

\begin{proof}
  By Theorem~7.5 in~\cite{Steinberg} every semisimple element of $G(\flb)$ is $G^0(\flb)$-conjugate to an element of
  $\cT(\flb)$.  The converse is clear as $\cT \cap G^0 = T$, an element $g \in G(\flb)$ is semisimple iff $g$ is of order
  prime to~$l$, and $l$ does not divide $(G:G^0)$.  Next we have $G = G^0 \cT$ since Borel subgroups in $G^0$ are conjugate
  and maximal tori in $B$ are conjugate. Consider a fixed coset $G^0h$ with $h \in \cT(\flb)$. By Lemma~4 of~\cite{Spr06} the
  elements $g(th)g^{-1} = [gt(hgh^{-1})^{-1}]h$ of $G^0 h$, where $t$ runs over $T(\flb)$ and $g$ runs over $G^0(\flb)$, are
  Zariski dense in $G^0h$.  (Lemma~4 of~\cite{Spr06} does not immediately apply to $h$ as $h$ is not a diagram automorphism.
  However for some $s \in T(\flb)$ the automorphism $g \mapsto s hgh^{-1} s^{-1}$ is a diagram automorphism and hence the
  elements $g t (hgh^{-1})^{-1} = g ts^{-1} (s hgh^{-1} s^{-1})^{-1} s$ as $t$ runs over $T(\flb)$ and $g$ runs over $G^0(\flb)$
  are Zariski dense in $G^0$.) Thus the $G^0(\flb)$-conjugates of $\cT(\flb)$ are Zariski dense in $G(\flb)$.
  For the last claim note that if $\tr(gw) = 0$ for some $w \in \ad V$ and some Zariski dense subset of $g \in G(\flb)$, then $w = 0$.
\end{proof}

The proof of our main theorem relies on Proposition~\ref{prop:alg-rep} and thus on the classification of finite simple groups.
(It still holds without it for $l$ sufficiently large, depending on~$d$ and ineffective, due to the results of Larsen and Pink~\cite{LP}.)

\begin{thm}\label{thm:main}
  Suppose that $V$ is a finite-dimensional $\flb$-vector space and that $\Gamma \subset \GL(V)$ is a finite subgroup that
  acts irreducibly on $V$. Let $\Gamma^0 \subset \Gamma$ be the subgroup generated by elements of $l$-power order. Then $V$
  is a semisimple $\Gamma^0$-module. Let $d \geq 1$ be the maximal dimension of an irreducible $\Gamma^0$-submodule of $V$.
  Suppose that $l \geq 2(d+1)$. Then:
  \begin{enumerate}
  \item $H^0(\Gamma, \ad^0 V) = H^1(\Gamma, \ad^0 V) = H^1(\Gamma,\flb)=0.$
  \item The set $\Gamma\ss$ spans $\ad V$ as an $\flb$-vector space.
 \end{enumerate}
 In particular, for any finite subfield $k$ of $\flb$ containing the eigenvalues of all elements of $\Gamma$ and such that
 $\Gamma \subset \GL_n(k)$, $\Gamma$ is adequate.
\end{thm}

\begin{proof}
Write $V = \bigoplus_i {W_i}$ as a direct sum of irreducible $\Gamma^0$-modules. Note that $\Gamma/\Gamma^0$ has order prime to~$l$.

We claim that $\dim V$ is prime to~$l$. Let $U$ be an irreducible constituent of $V$ as a $\Gamma^0$-module and let $V'$ be
the $U$-isotypic direct summand of $V$. Since $\Gamma$ acts transitively on the set of isotypic components and as
$(\Gamma:\Gamma^0)$ is prime to~$l$, it suffices to show that $\dim V'$ is prime to~$l$.  Let $\Gamma' \supset \Gamma^0$ be
the stabiliser of $V'$. Then $V'$ is an irreducible $\Gamma'$-module.  By Theorem 51.7 in \cite{curtis-reiner},
$U$ extends to a projective representation of $\Gamma'$ and there is an irreducible projective
representation $U'$ of $\Gamma'/\Gamma^0$ such that $V' \cong U \otimes U'$ (as projective $\Gamma'$-representation). The
claim follows as $\dim U < l$ and $\Gamma'/\Gamma^0$ is of order prime to~$l$.

By Proposition~\ref{prop:alg-rep} there exists an algebraic group $G = G^0 \rtimes H$ over~$\fl$ and a semisimple 
representation $r : G_{/\flb} \to \GL(V)$, where $G^0$ is connected simply connected semisimple, $H$ is a finite group of
order prime to $l$, and $r(G(\fl)) = \Gamma$. Moreover $\Gamma$ has no composition factor of order~$l$, which implies
that no quotient of $\Gamma^0$ contains a non-trivial normal $l$-subgroup.

We have
\[H^1(\Gamma, \ad V) = \bigoplus_{i,j} H^1(\Gamma^0, \Hom(W_i, W_j))^\Gamma\]
and
\[H^1(\Gamma^0, \Hom(W_i, W_j)) = \Ext^1_{\Gamma^0}(W_i, W_j),\]
which vanishes by \cite{Gur99}, Theorem A, since $\dim W_i +
\dim W_j \leq l-2$. (We apply that theorem to the quotient of $\Gamma^0$ that acts faithfully. Note that we saw above
that this quotient does not have a non-trivial normal $l$-subgroup.)
Similarly, $2 \leq l- 2$ implies that $H^1(\Gamma, \flb) = 0$.
Since $\dim V$ is prime to~$l$ it follows that $H^0(\Gamma,\ad^0 V) = 0$ and that $\ad^0 V$ is a direct summand of $\ad V$,
so $H^1(\Gamma,\ad^0 V) = 0$. This proves the first part above.

Let $G^0 \supset B \supset T$ denote a Borel and maximal torus defined
over~$\fl$. Proposition~\ref{prop:alg-rep} also shows that $|\ang{\mu,\alpha^\vee}| < (l-1)/2$ for all weights $\mu$
of~$T_{/\flb}$ on~$V$ and all $\alpha \in \Delta$. In particular, all dominant weights of $T_{/\flb}$ on~$V$ and~$\ad V$ are restricted.
Note that if $W$ is a semisimple $G^0_{/\flb}$-module such that all dominant weights of $T_{/\flb}$ on~$W$ are restricted, then 
every $G^0({\fl})$-submodule of~$W$ is also a $G^0_{/\flb}$-submodule. We apply this first to $V$
(which is semisimple as $G^0_{/\flb}$-module, since $r$ is semisimple), so the $W_i$ are $G^0_{/\flb}$-submodules. By Proposition~8 of \cite{Ser94} we see that $\ad V = \bigoplus_{i,j} \Hom(W_i,W_j)$ is a semisimple $G^0_{/\flb}$-module. (Note
that $\dim W_i + \dim W_j <l+2$.) Thus every $G^0({\fl})$-submodule of~$\ad V$ is also a $G^0_{/\flb}$-submodule.

By Lemma~\ref{lm:torus} (applied with $\Delta_*$ the set of simple coroots), the $\flb$-linear span of the image of
$T(\fl)$ in $\ad V$ equals the $\flb$-linear span of the image of $T(\flb)$.
Thus the $G^0({\fl})$-submodule of~$\ad V$ generated by the $\flb$-linear span of~$r(H)$ equals the
$G^0(\flb)$-submodule generated by $r(T(\flb)H)$.  By Lemma~\ref{lm:ss-elts-span} (noting that $\cT(\flb) = T(\flb)H$)
it follows that $r(H)$ spans $\ad V$.  As $r(H) \subset \Gamma\ss$, this completes the proof.
\end{proof}

\bibliographystyle{alpha}
\bibliography{ModLiftBib}

\end{document}